\documentclass[11pt,reqno]{amsart}

\numberwithin{equation}{section}

\usepackage{color,graphics,epsfig,a4wide}

\newcommand{\calD}{\mathcal{D}}

\newcommand{\calL}{\mathcal{L}}

\newcommand{\calS}{\mathcal{S}}

\newcommand{\calW}{\mathcal{W}}

\newcommand{\mA}{\mathbb{A}}

\newcommand{\mC}{\mathbb{C}}
\newcommand{\mD}{\mathbb{D}}

\newcommand{\mR}{\mathbb{R}}

\newcommand{\mZ}{\mathbb{Z}}

\newtheorem{theorem}{Theorem}[section]

\newtheorem{proposition}[theorem]{Proposition}

\theoremstyle{definition}

\newtheorem{remark}[theorem]{Remark}
\newtheorem{example}[theorem]{Example}

\theoremstyle{definition}
\newtheorem{definition}[theorem]{Definition}

\theoremstyle{definition}

\begin{document}

\keywords{partial differential equations, 
  distributions that are periodic in the spatial directions, Fourier transformation, 
behaviours, autonomy, controllability, spatially invariant systems}

\subjclass[2010]{Primary 35A24; Secondary 93B05, 93C20}

\title[Behaviours of spatially invariant systems]{Algebraic 
characterization of autonomy and controllability  of behaviours of 
spatially invariant systems}

\author{Amol Sasane}
\address{Department of Mathematics, London School of Economics,
     Houghton Street, London WC2A 2AE, United Kingdom.}
\email{sasane@lse.ac.uk}

\begin{abstract} 
  We give algebraic characterizations of the properties 
of autonomy and of controllability of behaviours of 
spatially invariant dynamical systems, consisting of distributional solutions $w$,  
that are periodic in the spatial variables, to a system of partial differential equations   
$$
M\left(\frac{\partial}{\partial x_1},\cdots, \frac{\partial}{\partial x_d} , 
\frac{\partial}{\partial t}\right) w=0,
$$
corresponding to a polynomial matrix  $M\in (\mC[\xi_1,\dots, \xi_d, \tau])^{m\times n}$. 
\end{abstract}

\maketitle

\section{Introduction}

Consider a homogeneous, linear, constant coefficient partial differential equation, 
in $\mR^{d+1}$ described by a polynomial matrix $p\in \mC[\xi_1,\dots, \xi_d, \tau]$:
\begin{equation}
\label{eq_ker_rep_p}
p\left(\frac{\partial}{\partial x_1},\cdots, \frac{\partial}{\partial x_d} , 
\frac{\partial}{\partial t}\right) w=0.
\end{equation}
That is, the differential operator 
$$
p\left( \displaystyle 
\frac{\partial}{\partial x_1}, \cdots, \frac{\partial}{\partial x_d}, \frac{\partial}{\partial t}\right)
$$  
is obtained from the polynomial $p\in \mC[\xi_1,\dots, \xi_d, \tau]$ by making the replacements
$$
\xi_k\rightsquigarrow\frac{\partial}{\partial x_k} \;\textrm{ for } k=1, \dots, d,\;\textrm{ and }\;\;
\tau\rightsquigarrow \frac{\partial}{\partial t}.
$$
More generally, given a polynomial {\em matrix} 
$M \in  (\mC[\xi_1,\dots, \xi_d, \tau])^{m\times p}$, consider the corresponding 
 {\em system} of PDEs
\begin{equation}
\label{eqn_PDE_system}
M\left( \displaystyle 
\frac{\partial}{\partial x_1}, \cdots, \frac{\partial}{\partial x_d}, \frac{\partial}{\partial t}\right)w:=
\left[\begin{array}{ccc}
     \displaystyle   \sum_{j=1}^n p_{1j} \left( \displaystyle 
\frac{\partial}{\partial x_1}, \cdots, \frac{\partial}{\partial x_d}, \frac{\partial}{\partial t}\right)w_{j} \\
\vdots \\
\displaystyle \sum_{j=1}^n p_{mj} \left( \displaystyle 
\frac{\partial}{\partial x_1}, \cdots, \frac{\partial}{\partial x_d}, \frac{\partial}{\partial t}\right)w_{j}
      \end{array}
 \right]
=
0,
\end{equation}
where solutions $w$ now have the $n$ components $w_1, \dots, w_n$, and $M=[p_{ij}]$ with $p_{ij}$ denoting the 
polynomial entries of $M$ for $1\leq i\leq m$ and $1\leq j \leq n$. 

In the behavioural approach to control theory pioneered by Willems \cite{PolWil}, the  ``behaviour'' ${\mathfrak{B}}_{\calW}(M)$ 
associated with $M$  in $\calW^n$ (where $\calW$ is an appropriate solution space, 
for example smooth functions $C^\infty(\mR^{d+1})$ or distribution spaces like 
$\calD'(\mR^{d+1})$ or $\calS'(\mR^{d+1})$ and so on), is defined to be the set of all solutions 
$w\in \calW^n$ that satisfy the above PDE system \eqref{eqn_PDE_system}. The aim is then to obtain algebraic characterizations 
(in terms of algebraic properties of the polynomial matrix $M$) of 
certain analytical properties of ${\mathfrak{B}}_{\calW}(M)$ (for 
example, the control theoretic properties of autonomy, 
controllability, stability, and so on). We refer the reader to \cite{PolWil} for 
background on the behavioural approach in the case of systems of ordinary 
differential equations, and to  \cite{PilSha}, \cite{SasThoWil}, \cite{ObeSch} 
\cite{BalSta} for distinct takes on this in the context of systems described by partial differential equations. 

The goal of this article is to give algebraic characterizations of the properties 
of autonomy and of controllability of behaviours of 
spatially invariant dynamical systems, consisting of distributional solutions $w$,  
that are periodic in the spatial variables, to a system of partial differential equations   
$$
M\left(\frac{\partial}{\partial x_1},\cdots, \frac{\partial}{\partial x_d} , 
\frac{\partial}{\partial t}\right) w=0,
$$
corresponding to a polynomial matrix  $M\in (\mC[\xi_1,\dots, \xi_d, \tau])^{m\times n}$. 
We give the relevant definitions below, and also state our two main results in Theorem~\ref{theorem_matricial_autonomy} 
(characterizing autonomy) and Theorem~\ref{result_2} (characterizing controllability). 

\subsection{Autonomy} 

Let us first consider the property of ``autonomy'', which means the following. 

\begin{definition}
 Let $\calW$ be a subspace of $\calD'(\mR^{d+1})$ which is invariant under differentiation, that is, 
  for all $w\in \calW$,
$$
 \frac{\partial}{\partial x_k}w \in  \calW\; \textrm{ for } k=1, \dots, d,\; \textrm{ and }\;\; 
 \frac{\partial}{\partial t}w \in  \calW. 
$$
If $M\in (\mC[\xi_1, \dots, \xi_d,\tau])^{m\times n}$, then the {\em behaviour} ({\em of $M$ in $\calW$}) is 
$$
{\mathfrak{B}}_{\calW}(M):= \left\{w\in \calW^n: 
M\left( \frac{\partial}{\partial x_1},\cdots,\frac{\partial}{\partial x_d},  
\frac{\partial}{\partial t}\right) w=0\right\}.
$$
We call  the behaviour ${\mathfrak{B}}_{\calW}(M)$ {\em autonomous} ({\em with respect to $\calW$}) if the only 
$w\in {\mathfrak{B}}_{\calW}(M)$ 
satisfying $w|_{t<0}=0$ is $w=0$. 
\end{definition}

The following \cite[Theorem~3.4]{SasThoWil} is a consequence of \cite[p.310, Theorem~8.6.7]{Hor}. 

\begin{proposition}
 \label{prop_SasThoWil}
Let $p\in \mC[\xi_1, \dots, \xi_d,\tau]$ be a nonzero polynomial. 
Then the behaviour 
$$
{\mathfrak{B}}_{\calD'(\mR^{d+1})}(p):=
\left\{w\in \calD'(\mR^{d+1}): p\left( \frac{\partial}{\partial x_1},\cdots, \frac{\partial}{\partial x_d}, 
\frac{\partial}{\partial t}\right) w=0\right\}
$$ 
corresponding to $p$ is 
autonomous if and only if $\deg p = \deg p(\mathbf{0},\tau)$. 
\end{proposition}

Here by $\deg(\cdot)$, we mean the {\em total degree}, which
is the maximum (over the monomials occurring in the polynomial) of the
sum of the degrees of the exponents of indeterminates in the monomial. Also, 
$p(\mathbf{0},\tau)$ denotes the polynomial in $\mC[\tau]$ obtained from $p\in 
\mC[\xi_1, \dots, \xi_d,\tau]$ by making the substitutions 
$\xi_k \mapsto 0$ for $k=1,\dots, d$. 

There has been recent interest in ``spatially invariant systems'', see 
for example \cite{CurIftZwa}, \cite{CurSas}, where one considers 
solutions to PDEs that are periodic along the spatial direction. So it is a 
natural question to ask what the analogue of Proposition~\ref{prop_SasThoWil} is, when we replace the 
solution space $\calD'(\mR^{d+1})$ with one that consists only of those solutions that are 
{\em periodic} in the 
spatial directions. 

In this article, our first main result is the following one, 
characterizing autonomy of spatially invariant systems. 

\begin{theorem}
 \label{theorem_matricial_autonomy}
Suppose that $\mA:=\{{\mathbf{a_1}}, \dots, {\mathbf{a_d}}\}$ is a linearly independent set vectors in $\mR^d$. Let 
$M \in (\mC[\xi_1,\dots, \xi_d, \tau])^{m\times n}$ and let 
$$
{\mathfrak{B}}_{\calD'_{\mA}(\mR^{d+1})}(R):=
\left\{w\in (\calD'_{\mA}(\mR^{d+1}))^{n}: 
M\left( \frac{\partial}{\partial x_1},\cdots, \frac{\partial}{\partial x_d}, 
\frac{\partial}{\partial t}\right) w=0\right\}.
$$ 
Then the following two statements are equivalent:
\begin{enumerate}
 \item ${\mathfrak{B}}_{\calD'_{\mA}(\mR^{d+1})}(M)$ is autonomous. $\phantom{\Big(}$
 \item For all $\mathbf{v}\in A^{-1}\mZ^d$, 
 $
\displaystyle \max_{t \in \mC} \Big(\textrm{\em rank}\left( M( 2\pi i \mathbf{v}, t) \right) \Big)= n. $ 
\end{enumerate}
\end{theorem}

Here $\calD'_{\mA}(\mR^{d+1})$ is, roughly speaking, 
the set of all distributions on $\mR^{d+1}$ that are 
periodic in the spatial direction with a discrete set $\mA$ of periods. The precise definition of 
$ \calD'_{\mA}(\mR^{d+1})$ is given below in 
Subsection~\ref{subsection_sol_space}. 

We remark that our new result in Theorem~\ref{theorem_matricial_autonomy} is the multidimensional generalization 
of the result in \cite[\S 3.2]{PolWil}; see also \cite[Proposition~2.8.2]{Bel}. 

\begin{remark}
 \label{holmgren} Theorem~\ref{theorem_matricial_autonomy} can also be viewed as a version of 
the Holmgren Uniqueness Theorem. Indeed, a behaviour is autonomous if and only if the solution set of the 
system of PDEs \eqref{eqn_PDE_system} has the following uniqueness property: If a pair of solutions 
coincide in the past, then they are identical. In other words, whenever $w_1$, $w_2$ satisfying \eqref{eqn_PDE_system} are such that 
there exists a $\mathrm{T}\in \mR$ such that $w_1|_{t<\mathrm{T}}=w_2|_{t<\mathrm{T}}$, then $w_1=w_2$. 
\end{remark}

\subsection{The space $\calD'_{\mA}(\mR^{d+1})$} 
\label{subsection_sol_space}

 For ${\mathbf{a}}\in \mR^d$, the {\em translation operation} 
${\mathbf{S_a}}$ on distributions in $\calD'(\mR^d)$ is
defined by
$$
\langle {\mathbf{S_a}}(T),\varphi\rangle=\langle T,\varphi(\cdot+{\mathbf{a}})\rangle\;
\textrm{ for all }\varphi \in \calD(\mR^d).
$$
A distribution $T\in \calD'(\mR^d)$ is said to be {\em periodic with a period} 
$\mathbf{a}\in \mR^d$  if $T= {\mathbf{S_a}}(T)$.  

Let $\mA:=\{{\mathbf{a_1}}, \dots, {\mathbf{a_d}}\}$ be a linearly independent set vectors in $\mR^d$. 
We define $\calD'_{\mA}(\mR^d)$ to be the set of all distributions $T$ that satisfy 
$$
{\mathbf{S_{a_k}}}(T)=T, \quad k=1,\dots, d.
$$
From \cite[\S34]{Don}, $T$ is a tempered distribution, and 
from the above it follows by taking Fourier transforms that 
 $
(1-e^{2\pi i {\mathbf{a_k}} \cdot \mathbf{y}})\widehat{T}=0$ for $ k=1,\dots,d.
$ It can be seen that 
$$
\widehat{T}=\sum_{\mathbf{v} \in A^{-1} \mZ^d} \alpha_{\mathbf{v}}(T) \delta_{\mathbf{v}},
$$
for some scalars $\alpha_{\mathbf{v}}\in \mC$, and where  $A$ is the matrix with its rows equal to 
the transposes of the column vectors ${\mathbf{a_1}}, \dots, {\mathbf{a_d}}$:
$$
A:= \left[ \begin{array}{ccc} 
    \mathbf{a_1}^{\top} \\ \vdots \\ \mathbf{a_d}^{\top} 
    \end{array}\right].
$$
Also, in the above, $\delta_{\mathbf{v}}$ denotes the usual Dirac measure with support in $\mathbf{v}$:
$$
\langle \delta_{\mathbf{v}}, \psi\rangle =\psi (\mathbf{v}) \;\textrm{ for }\psi \in \calD'(\mR^d).
$$
By the Schwartz Kernel Theorem (see for instance
\cite[p.~128, Theorem~5.2.1]{Hor}),  $\calD'(\mR^{d+1})$ is isomorphic as a topological space to 
$\calL(\calD(\mR), \calD'(\mR^d))$, the space of all continuous linear maps from $\calD(\mR)$ to 
$\calD'(\mR^d)$, thought of as vector-valued distributions. For preliminaries on vector-valued distributions, 
we refer the reader to \cite{Car}. We indicate this isomorphism by putting an arrow on top of elements of 
$\calD'(\mR^{d+1})$. Thus for $w\in \calD'(\mR^{d+1})$, we set $\vec{w}\in \calL(\calD(\mR), \calD'(\mR^d))$ 
to be the vector valued distribution defined by 
$$
\langle \vec{w}(\varphi), \psi\rangle=\langle w, \psi \otimes \varphi\rangle 
$$
for $\varphi\in \calD(\mR)$ and $\psi \in \calD(\mR^d)$. We define 
$$ 
\calD'_{\mA}(\mR^{d+1})=\{w\in \calD'(\mR^{d+1}): \textrm{ for all }\varphi \in \calD(\mR), \;
\vec{w}(\varphi) \in \calD'_{\mA}(\mR^d)\}.
$$
Then for $w\in \calD'_{\mA}(\mR^{d+1})$, 
$$
 \frac{\partial}{\partial x_k}w \in  \calD'_{\mA}(\mR^{d+1}) \;\textrm{ for } k=1, \dots, d, \;
\textrm{ and } \;\; 
\frac{\partial}{\partial t}w \in  \calD'_{\mA}(\mR^{d+1}). 
$$
Also, for $w\in \calD'_{\mA}(\mR^{d+1})$, we define $\widehat{w}\in \calD'(\mR^{d+1})$ 
by 
$$
\langle \widehat{w}, \psi \otimes \varphi \rangle
=
\langle \vec{w}(\varphi), \widehat{\psi} \rangle,
$$
for $\varphi \in \calD(\mR)$ and $\psi \in \calD(\mR^d)$. That this specifies a well-defined 
distribution in $\calD'(\mR^{d+1})$, can be seen using the fact that for every $\Phi\in 
\calD(\mR^{d+1})$, there exists a sequence of functions $(\Psi_n)_n$ that are finite sums of 
direct products of test functions, that
is, $\Psi_n=\sum_{k} \psi_{k} \otimes \varphi_k$, where 
$\psi_k\in \calD(\mR^d)$ and  $\varphi_k\in \calD(\mR)$, such that $\Psi_n$ converges to $\Phi$  
in $\calD(\mR^{d+1})$. We also have 
$$
\widehat{
\frac{\partial}{\partial x_k} w}= 2\pi i y_k \widehat{w} \;\textrm{ for } k=1, \dots, d,\;
\textrm{ and }\;\;
\widehat{
\frac{\partial}{\partial t} w}=\frac{\partial }{\partial t}\widehat{w}.
$$
Here $\mathbf{y}=(y_1,\dots, y_d)$ is  the Fourier transform variable. 

\subsection{Controllability}

Next we consider the property of controllability for a behaviour.

\begin{definition}
Let $\calW$ be a subspace of $(\calD'(\mR^{d+1}))^n$ which is invariant under differentiation.  
For $M\in (\mC[\xi_1, \dots, \xi_d,\tau])^{m\times n}$, we call its behaviour ${\mathfrak{B}}_{\calW}(M)$ in $\calW$ 
{\em controllable} if for every $w_1, w_2 \in {\mathfrak{B}}_{\calW}(M)$, there 
is a $w\in {\mathfrak{B}}_{\calW}(M)$ and a $\textrm{T} \geq 0$ such that  $
w|_{(-\infty, 0)}=w_1|_{(-\infty,0)}$  and $ w|_{(\textrm{T} , +\infty)} =w_2|_{(\textrm{T} , +\infty)} $. 
\end{definition}

The second main result in this article is the following one, characterizing 
controllability of spatially invariant systems:

\begin{theorem}
\label{result_2}
Suppose that $\mA:=\{{\mathbf{a_1}}, \dots, {\mathbf{a_d}}\}$ is a linearly independent set vectors in $\mR^d$. Let 
$M \in (\mC[\xi_1,\dots, \xi_d, \tau])^{m\times n}$ and let 
$$
{\mathfrak{B}}_{\calD'_{\mA}(\mR^{d+1})}(M):=
\left\{w\in (\calD'_{\mA}(\mR^{d+1}))^{n}: 
M\left( \frac{\partial}{\partial x_1},\cdots, \frac{\partial}{\partial x_d}, 
\frac{\partial}{\partial t}\right) w=0\right\}.
$$ 
Then the following two statements are equivalent:
\begin{enumerate}
 \item ${\mathfrak{B}}_{\calD'_{\mA}(\mR^{d+1})}(M)$ is controllable.
 \item For  each 
$\mathbf{v}\in A^{-1}\mZ^d$, there exists a nonnegative integer $r_{\mathbf{v}}\leq \min\{n,m\}$  such that  all $t\in \mC$, 
$
\textrm{\em rank}\left( M( 2\pi i \mathbf{v}, t) \right) = r_{\mathbf{v}}. $ 
\end{enumerate}
\end{theorem}

We remark that our new result in Theorem~\ref{result_2} is the multidimensional generalization 
of \cite[Theorem~5.2.5]{PolWil}.

\section{Proof of Theorem~\ref{theorem_matricial_autonomy}}

Before we prove our main result, we illustrate the key idea behind the 
sufficiency of our  algebraic condition  for autonomy. Take a trajectory in 
the behaviour with a zero past. By taking Fourier transform with respect to the spatial variables, 
the partial derivatives with respect to the spatial variables are converted into the polynomial
coefficients $c_{ij}(2\pi i \mathbf{y} )$, where $\mathbf{y}$ is the vector of 
Fourier transform variables $y_1 ,\dots, y_d$. But the support of $\widehat{w}$ is carried 
on a family of lines, indexed by $\mathbf{n}\in\mZ^d$, in $\mR^{d+1}$ parallel 
to the time axis. So we obtain a family of ODEs,
parameterized by $\mathbf{n}\in\mZ^d$, and by ``freezing'' an $\mathbf{n}\in \mZ^d$, we get
an ODE, where for a solution we can indeed say that zero past implies 
a zero future, and so the proof can be completed easily.

\begin{proof}[Proof of Theorem~\ref{theorem_matricial_autonomy}] {\bf (1) $\Rightarrow$ (2)}: 
Let $\mathbf{v_0}\in A^{-1}\mZ^d$ be such that 
$$
\displaystyle \max_{t \in \mC} \Big(\textrm{rank}(M(2\pi i \mathbf{v_0}, t))\Big) <n.
$$ 
Then from \cite[\S 3.2]{PolWil}, it follows that the smooth behaviour 
${\mathfrak{B}}_{C^{\infty}(\mR)}(M(2\pi i \mathbf{v_0},\tau))$ of the system of ODEs corresponding 
to the polynomial matrix $M(2\pi i \mathbf{v_0},\tau) \in (\mC[\tau])^{m\times n}$ is {\em not} autonomous. 
This means that there is a nonzero $\Theta \in {\mathfrak{B}}_{C^{\infty}(\mR)}(M(2\pi i \mathbf{v_0},\tau))$ 
such that $\Theta|_{t<0}=0$. 

Define $w:=e^{2\pi i{\mathbf{v_0}} \cdot \mathbf{x}}\otimes \Theta$. Here  $\mathbf{v_0}\cdot\mathbf{x}$ is the usual Euclidean 
inner product in the complex vector space $\mC^d$ of $\mathbf{v_0}$ and $\mathbf{x}$. 
Then we have that $w\in (\calD'_{\mA}(\mR^{d+1}))^n$, since  
$$
\mathbf{S_{a_k}} w= e^{2\pi i \mathbf{v_0}\cdot (\mathbf{x}+\mathbf{a_k})}\otimes \Theta
=e^{2\pi i \mathbf{v_0}\cdot \mathbf{a_k}}e^{2\pi i \mathbf{v_0}\cdot \mathbf{x}}\otimes \Theta
= 1\cdot e^{2\pi i\mathbf{v_0}\cdot \mathbf{x}}\otimes \Theta= w.
$$
Moreover, $w$ has 
zero past, that is, $w|_{t<0}=0$ because $\Theta|_{t<0}=0$. 
Also, the trajectory $w\in {\mathfrak{B}}_{\calD'_{\mA}(\mR^{d+1})}(M)$ because
 $$
M\left(\frac{\partial}{\partial x_1},\cdots, \frac{\partial}{\partial x_d} , \frac{\partial}{\partial t}\right) 
 w= e^{2\pi i \mathbf{v_0}\cdot \mathbf{x}}  M\left( 2\pi i \mathbf{v_0} , \frac{d}{dt} \right)  \Theta
=e^{2\pi i \mathbf{v_0}\cdot \mathbf{x}} \cdot 0=0.
$$ 
Consequently,  ${\mathfrak{B}}_{\calD'_{\mA}(\mR^{d+1})}(M)$ is not autonomous. This 
completes the proof of the `only if' part. 

\bigskip 

\noindent  {\bf (2) $\Rightarrow$ (1)}: On the other hand, now suppose that for each 
$\mathbf{v}\in A^{-1}\mZ^d$, 
$$
\displaystyle \max_{t \in \mC} \Big(\textrm{rank}(M(2\pi i \mathbf{v}, t))\Big) =n.
$$ 
 Then 
from  \cite[\S 3.2]{PolWil}, it follows that the distributional behaviour 
${\mathfrak{B}}_{\calD'(\mR)}(M(2\pi i \mathbf{v},\tau))$ 
of the system of ODEs corresponding to the polynomial matrix 
$M(2\pi i \mathbf{v},\tau) \in (\mC[\tau])^{m\times n}$ is autonomous. Thus for each 
$\mathbf{v}\in A^{-1}\mZ^d$ and for a distribution 
$T\in (\calD'(\mR))^n$ such that $T|_{t<0}=0$ and 
$$
M\left(2\pi i \mathbf{v},\frac{d}{dt}\right)T=0,
$$
there holds that $T=0$. 

Now suppose that the trajectory $w\in (\calD'_{\mA}(\mR^{d+1}))^n$ satisfies 
$w|_{t<0}=0$ and 
\begin{equation}
 \label{eq_pf_if_part}
M\left(\frac{\partial}{\partial x_1},\cdots, \frac{\partial}{\partial x_d} , \frac{\partial}{\partial t}\right) 
 w=0.
\end{equation} 
 Upon taking Fourier transformation on both
sides of the equation \eqref{eq_pf_if_part} with respect to the spatial
variables, we obtain
\begin{equation}
\label{eq_ast}
M\left(2\pi i \mathbf{y}, \frac{\partial}{\partial t} \right) \widehat{w} =0.
\end{equation}
For each fixed $\varphi \in \calD(\mR)$, $\vec{w}(\varphi)\in (\calD_{\mA}'(\mR^d))^{n}$, and so it 
follows that 
\begin{equation}
\label{eq_ast2}
\widehat{w}=\sum_{\mathbf{v} \in A^{-1} \mZ^d} \delta_{\mathbf{v}} \;\! \alpha_{\mathbf{v}}(\widehat{w},\varphi) ,
\end{equation}
for appropriate vectors $\alpha_{\mathbf{v}}(\widehat{w},\varphi)\in \mC^n$. In particular, 
we see that the support of $\widehat{w}$ is contained in $A^{-1} \mZ^d \times [0,+\infty)$. 
Thus each of the half lines in $A^{-1} \mZ^d \times [0,+\infty)$ carries a solution of the differential
equation \eqref{eq_ast}, and $\widehat{w}$ is a sum of these. We will
show that each of these summands is zero. It can be seen from \eqref{eq_ast2} that the map 
$\varphi \mapsto  \alpha_{\mathbf{v}}(\widehat{w},\varphi) :\calD(\mR) \rightarrow \mC^n$ defines a vector distribution 
$T^{(\mathbf{v})}$ in $(\mD'(\mR))^n$. Moreover, the support of $T^{(\mathbf{v})}$ is contained in $[0,+\infty)$.  
From \eqref{eq_ast2}, we see that for a small enough neighbourhood $N$ of $\mathbf{v} \in A^{-1} \mZ^d$ in $\mR^d$, we have 
$$
\delta_{\mathbf{v}} \otimes M\left(2\pi i \mathbf{v}, \frac{d}{d t} \right)  T^{(\mathbf{v})}=0
$$
in $N\times \mR$. 
But as we had seen above, our algebraic hypothesis implies that the behaviour 
$\mathfrak{B}_{\calD'(\mR)} (M(2\pi i \mathbf{v},\tau))$ is autonomous, and so $T^{(\mathbf{v})}=0$. 
As this happens with each $\mathbf{v}\in A^{-1}\mZ^d$, we conclude that $\widehat{w}=0$ and hence  
also $w=0$. Consequently, the behaviour ${\mathfrak{B}}_{\calD'_{\mA}(\mR^{d+1})}(p)$ 
is autonomous.
\end{proof}

\begin{example}[Diffusion equation]\label{example_diffusion} Consider the diffusion equation 
$$
\left( \frac{\partial}{\partial t}- 
\frac{\partial^2}{\partial x_1^2}-\cdots -\frac{\partial^2}{\partial x_d^2}\right)w=0,
$$
corresponding to to the polynomial $p=\tau-\xi_1^2-\cdots -\xi_d^2$. 
For each $\mathbf{v} \in A^{-1}\mZ^d$, we have that 
$$
\textrm{rank}(p(2\pi i \mathbf{v}, t))= \textrm{rank}(t+ 4\pi^2 \|\mathbf{v}\|_2^2) = \left\{ \begin{array}{ll}
                                                           0 & \textrm{if } t = -4\pi^2 \|\mathbf{v}\|_2^2,\\
                                                           1 & \textrm{if } t \neq  -4\pi^2 \|\mathbf{v}\|_2^2,
                                                          \end{array}\right. 
$$
Thus  
$$
\max_{t \in \mC} \Big( \textrm{rank}(p(2\pi i \mathbf{v}, t))\Big) =1\; (=n),
$$
and so it follows  from Theorem~\ref{theorem_matricial_autonomy} that 
the behaviour ${\mathfrak{B}}_{\calD'_{\mA}(\mR^{d+1})}(p)$ is always autonomous, no matter 
what $\mA$ is.

Note that this is in striking contrast to what happens when we look at just 
distributional solutions: since
$$
\begin{array}{r}
\!\!\deg(p(\xi_1,\dots, \xi_d,\tau))=\deg(\tau - (\xi_1^2+\dots+
\xi_d^2))=2\\
\;\!\mathrel{\rotatebox[origin=c]{90}{$\neq$}} \\
\!\!\deg(p(0,\dots, 0,\tau))=\deg(\tau)=  1
\end{array}
$$
we have from Proposition~\ref{prop_SasThoWil} that ${\mathfrak{B}}_{\calD'(\mR^{d+1})}(p)$ is {\em not} autonomous, 
and this  outcome is physically unexpected. Indeed, if
we imagine the case of diffusion of heat, in which case the $w$ is
the temperature, say along a metallic rod when $d=1$, then zero temperature up to time $t=0$ should mean
that the temperature stays zero in the future as well (since the above PDE
describes the situation when no external heat is supplied).
However, when one considers distributional solutions, one can have pathological 
solutions with a zero past that are nonzero in the future! But if we choose the
physically ``correct'' solution space in this context, namely
functions which at each time instant have a spatial profile belonging
to $L^\infty(\mR)$, then it can be shown that solutions that are zero in the past are also zero in the future, as
expected.  So the real reason for the nonautonomy when one considers solutions in 
 $\calD'(\mR^{d+1})$ is that there is no
restriction on the spatial profiles of the solutions at each time
instant, and wild growth (such as something which grows faster than
$e^{|\mathbf{x}|^2}$) is allowed.  However, with a {\em periodic} profile in the 
spatial direction, namely when the spatial profile is in  $\calD'_{\mA}(\mR^{d})$, 
we know that the spatial profile is automatically tempered (see for example \cite{Don}), and as we have seen above, 
 in this case 
the behaviour  ${\mathfrak{B}}_{\calD'_\mA(\mR^{d+1})}(p)$ {\em is} autonomous, in conformity 
with our physical expectation.
\hfill$\Diamond$
\end{example}

\section{Proof of Theorem~\ref{result_2}}

\begin{proof}[Proof of Theorem~\ref{result_2}] {\bf (1) $\Rightarrow$ (2)}: 
Suppose that  $ {\mathfrak{B}}_{ \calD'_{\mA}(\mR^{d+1}) }(M)$ is controllable. Moreover let 
there exist a vector $\mathbf{v_0}\in A^{-1}\mZ^d$ for which  it is not the case that 
the rank of $R(2\pi i {\mathbf{v_0}}, t)$ is the same for all $t\in \mC$. 
Then we know that the behaviour ${\mathfrak{B}}_{\calD'(\mR)} (M(2\pi i {\mathbf{v_0}}, \tau))$ 
in $(\calD'(\mR))^{n}$ associated 
with the system of ODEs given by $M(2\pi i {\mathbf{v_0}}, \tau) \in (\mC[\tau])^{m\times p}$ is not 
controllable (as can be seen from \cite[Theorem~2, p.396]{PilSha} in the special case of ODEs). Suppose that 
 $\Theta_1, \Theta_2 \in  
{\mathfrak{B}}_{\calD'(\mR)} (M(2\pi i {\mathbf{v_0}}, \tau))$. Set 
\begin{eqnarray*}
 w_1&:=& e^{2\pi i {\mathbf{v_0}}\cdot \mathbf{x}}  \otimes \Theta_1,\\
 w_2&:=& e^{2\pi i {\mathbf{v_0}}\cdot \mathbf{x}}  \otimes \Theta_2. 
\end{eqnarray*}
Then $w_i \in (\calD'_{\mA}(\mR^{d+1}))^{n}$, $i=1,2$,  since  for all $k\in \{1,\dots, d\}$, we have 
 $$
\mathbf{S_{a_k}} w_i= e^{2\pi i \mathbf{v_0}\cdot (\mathbf{x}+\mathbf{a_k})}\otimes \Theta_i
=e^{2\pi i \mathbf{v_0}\cdot \mathbf{a_k}}e^{2\pi i \mathbf{v_0}\cdot \mathbf{x}}\otimes \Theta_i
= 1\cdot e^{2\pi i \mathbf{v_0}\cdot \mathbf{x}}\otimes \Theta_i= w_i.
$$
Also, $w_i\in {\mathfrak{B}}_{\calD'_{\mA}(\mR^{d+1})}(M)$, $i=1,2$, because
$$
M\left(\frac{\partial}{\partial x_1},\cdots, \frac{\partial}{\partial x_d} , \frac{\partial}{\partial t}\right) 
 w_i
= e^{2\pi i \mathbf{v_0}\cdot \mathbf{x}} R\left(2\pi i {\mathbf{v_0}}, \frac{d}{dt}\right)\Theta_i
=0\cdot e^{2\pi i \mathbf{v_0} \cdot \mathbf{x}}=0,
$$
thanks to the fact that 
$$
M\left(2\pi i {\mathbf{v_0}}, \frac{d}{dt}\right) \Theta_i=0.
$$ 
Suppose $w$ patches up $w_1$ and $w_2$, 
that is,  $w\in {\mathfrak{B}}_{\calD'_{\mA}(\mR^{d+1})}(M)$ satisfies 
$w|_{(-\infty, 0)}
=w_1|_{(-\infty,0)}$ and $w|_{(\textrm{T}, +\infty)}=w_2|_{(\textrm{T}, +\infty)} \neq 0$, for some $\textrm{T}\geq 0$. Then 
$$
 \widehat{w} = \sum_{\mathbf{v} \in A^{-1} \mZ^d} \delta_{\mathbf{v}}\otimes T^{(\mathbf{v})},
$$
for some $T^{(\mathbf{v})} \in (\calD'(\mR))^n$. Thus we have 
$$
M\left(2\pi i {\mathbf{v_0}}, \frac{d}{dt}\right) T^{(\mathbf{v_0})}=0,
$$ 
and $T^{(\mathbf{v_0})}|_{(-\infty,0)}=\Theta_1|_{(-\infty,0)}$, while $T^{(\mathbf{v_0})}|_{(\mathrm{T}, +\infty)}
=\Theta_2|_{(\mathrm{T}, +\infty)}$. Thus we have shown that for every $\Theta_1, \Theta_2 \in 
{\mathfrak{B}}_{\calD'(\mR)} (M(2\pi i {\mathbf{v_0}}, \tau))$, there exists a $\mathrm{T}\geq 0$ and a 
$T^{(\mathbf{v_0})} \in {\mathfrak{B}}_{\calD'(\mR)} (M(2\pi i {\mathbf{v_0}}, \tau))$ 
such that $T^{(\mathbf{v_0})}|_{(-\infty,0)}=\Theta_1|_{(-\infty,0)}$, while $T^{(\mathbf{v_0})}|_{(\mathrm{T}, +\infty)}
=\Theta_2|_{(\mathrm{T}, +\infty)}$. In other words, the behaviour 
${\mathfrak{B}}_{\calD'(\mR)} (M(2\pi i {\mathbf{v_0}}, \tau))$ is 
controllable, a contradiction.  This completes the proof of the fact that
 (1) $\Rightarrow$ (2). 

\bigskip 

\noindent  {\bf (2) $\Rightarrow$ (1)}: Let us suppose that $w_1, w_2 \in 
{\mathfrak{B}}_{\calD'_{\mA}(\mR^{d+1})}(M)$. Then   we have 
$$
 \widehat{w_1} = \sum_{\mathbf{v} \in A^{-1} \mZ^d} \delta_{\mathbf{v}}\otimes T_{1}^{(\mathbf{v})},
 \quad \widehat{w_2} = \sum_{\mathbf{v} \in A^{-1} \mZ^d} \delta_{\mathbf{v}}\otimes T_{2}^{(\mathbf{v})},
$$ 
for some $T_1^{(\mathbf{v})}, T_2^{(\mathbf{v})} \in (\calD'(\mR))^n$. Moreover, owing to the correspondence between 
$\calD'_{\mA}(\mR^d)$ and the space of sequences $s'(\mZ^d)$ of at most polynomial growth, it follows that for each $\varphi \in 
\calD(\mR)$, there exist $M_\varphi >0$ and a positive integer $k_\varphi$, such that we have the estimates 
$$
\|\langle T_{1}^{(\mathbf{v})}, \varphi\rangle\|_2  \leq  M_{\varphi}(1+\|\mathbf{n}\|_2)^{k_{\varphi}} ,\quad 
\|\langle T_{2}^{(\mathbf{v})} , \varphi\rangle\|_2  \leq  M_{\varphi}(1+\|\mathbf{n}\|_2)^{k_{\varphi}} ,
$$
for all $\mathbf{n}:=A\mathbf{v} \in \mZ^d$, and where $\|\cdot\|_2$ is the usual Euclidean norm.  Let $\mathrm{T}>0$, and 
let $\theta \in C^\infty(\mR)$ be such that $\theta(t)=1$ for all $t\leq 0$, $\theta(t)=0$ for all $t>\mathrm{T}/4$ 
and $0\leq \theta (t)\leq 1$ for all $t\in \mR$. Define $T^{(\mathbf{v})}\in (\calD'(\mR))^n$ by 
$$
T^{(\mathbf{v})}:= \theta T_1^{(\mathbf{v})}+\theta(\mathrm{T}-\cdot) T_2^{(\mathbf{v})}.
$$
Set $\widehat{w}\in \calD'(\mR^{d+1})$ to be 
 $
\widehat{w}=\displaystyle \sum_{\mathbf{v} \in A^{-1} \mZ^d}\delta_{\mathbf{v}}\otimes T^{(\mathbf{v})}.
$ 
Then for every $\varphi \in \calD(\mR)$, we have 
\begin{eqnarray*}
\|\langle T^{(\mathbf{v})}, \varphi\rangle\|_2  &\leq& \|\langle \theta T_1^{(\mathbf{v})}, \varphi \rangle \|_2+ 
\|\langle \theta(\mathrm{T}-\cdot) T_2^{(\mathbf{v})}, \varphi \rangle \|_2 \\
&\leq & M_{\theta \varphi} (1+\|\mathbf{n}\|_2)^{k_{\theta \varphi}}+ 
M_{\theta (\mathrm{T}-\cdot)\varphi} (1+\|\mathbf{n}\|_2)^{k_{\theta(\mathrm{T}-\cdot) \varphi}}\\
&\leq & \max\{M_{\theta \varphi},M_{\theta (\mathrm{T}-\cdot)\varphi}\}  
(1+\|\mathbf{n}\|_2)^{\max\{k_{\theta \varphi},k_{\theta(\mathrm{T}-\cdot) \varphi} \}},
\end{eqnarray*}
and so $\vec{w}(\varphi)\in \calD'_{\mA}(\mR^d)$.  Thus $w\in \calD'_{\mA}(\mR^{d+1})$. 
Also, $w\in {\mathfrak{B}}_{\calD'_{\mA}(\mR^{d+1})}(M)$ because 
$$
M\left( 2\pi i \mathbf{y}, \frac{\partial}{\partial t} \right) ( \delta_{\mathbf{v}}\otimes T^{(\mathbf{v})} ) 
= 
M\left( 2\pi i \mathbf{v}, \frac{d}{d t} \right)( \delta_{\mathbf{v}}\otimes T^{(\mathbf{v})} ) 
=
0,
$$
for each $\mathbf{v} \in A^{-1}\mZ^d$, and so 
 $$
 M\left( 2\pi i \mathbf{y}, \frac{\partial}{\partial t} \right) \widehat{w} =0.
$$ 
Consequently, 
$$
M\left( \frac{\partial}{\partial x_1},\cdots, \frac{\partial}{\partial x_d} , \frac{\partial}{\partial t} \right) w =0,
$$
that is, $w\in {\mathfrak{B}}_{\calD'_{\mA}(\mR^{d+1})}(M)$.

Finally, because 
$
T^{(\mathbf{v})}|_{(-\infty,0)}=T_1^{(\mathbf{v})}|_{(-\infty,0)}$ and $ 
T^{(\mathbf{v})}|_{(\mathrm{T}, +\infty)}=T_2^{(\mathbf{v})}|_{(\mathrm{T}, +\infty)}$, 
it follows that 
$\widehat{w}|_{(-\infty, 0)}=\widehat{w_1}|_{(-\infty,0)}$ and $\widehat{w}|_{(\mathrm{T}, +\infty)}
=\widehat{w_2}|_{(\mathrm{T}, +\infty)}$. Consequently, 
$w|_{(-\infty, 0)}=w_1|_{(-\infty,0)}$   and $w|_{(\mathrm{T}, +\infty)}=w_2|_{(\mathrm{T}, +\infty)}$, showing that the 
behaviour is controllable. This completes the proof.
\end{proof}

\begin{example} Let us consider again 
the polynomial $p=\tau-(\xi_1^2+\dots+\xi_d^2)$ from Example~\ref{example_diffusion}. 
For each $\mathbf{v} \in A^{-1}\mZ^d$, we have that 
$$
\textrm{rank}(p(2\pi i \mathbf{v}, t))= \textrm{rank}(t+ 4\pi^2 \|\mathbf{v}\|_2^2) = \left\{ \begin{array}{ll}
                                                           0 & \textrm{if } t = -4\pi^2 \|\mathbf{v}\|_2^2,\\
                                                           1 & \textrm{if } t \neq  -4\pi^2 \|\mathbf{v}\|_2^2,
                                                          \end{array}\right. 
$$
and so it follows  from Theorem~\ref{result_2} that  ${\mathfrak{B}}_{\calD'_{\mA}(\mR^{d+1})}(p)$ 
is not controllable. 
\hfill$\Diamond$
\end{example}

\begin{remark}
We mention a related question: Is there an algebraic characterization in terms of $M$ of 
{\em approximate} controllability of ${\mathfrak{B}}_{\calD'_{\mA}(\mR^{d+1})}(M)$?  Here, 
by approximate controllability, we mean the following.

Let 
 $ 
\calD'_{\mA}(\mR^{d}\times (0,+\infty)):=\{w|_{t>0}: w \in \calD'_{\mA}(\mR^{d+1})\} ,
$  
endowed with the induced topology from $\calD'(\mR^d\times (0,+\infty))$. 
We call ${\mathfrak{B}}_{\calD'_{\mA}(\mR^{d+1})}(M)$ 
{\em approximately controllable} if for every pair of trajectories $w_1, w_2 \in {\mathfrak{B}}_{\calD'_{\mA}(\mR^{d+1})}(M)$, 
and every neighbourhood $N$ 
of $0$ in $ 
(\calD'_{\mA}(\mR^{d}\times (0,+\infty)))^n,
$ 
there 
is a $w\in {\mathfrak{B}}_{\calD'_{\mA}(\mR^{d+1})}(M)$ and a $\mathrm{T}\geq 0$ such that  
\begin{eqnarray*}
&& w|_{(-\infty, 0)}=w_1|_{(-\infty,0)}\;\textrm{  and }\\
&&\mathbf{S_{(\mathbf{0},-\mathrm{T})} } (w-w_2)|_{(0, +\infty)} \in N.
\end{eqnarray*}
 Here,  $\mathbf{S_{(\mathbf{0},-\mathrm{T})} } $ denotes the translation operator in $(\calD'(\mR^{d+1}))^n$ corresponding to 
 the vector $(\mathbf{0},-\mathrm{T})\in \mR^{d+1}$.

When one considers just {\em smooth} solutions, that is, $\calW=C^\infty(\mR^{d+1})$, 
then one can define approximate controllability analogously to the above. 
(In particular, then the $N$ in the definition  is  a neighbourhood in the appropriate topology 
of the Frechet space of $C^\infty$ functions in the half-space $t>0$.) Using the main result in \cite{Hor0}, 
a  characterization of approximate controllability 
in the smooth solution case was given in \cite[Theorem~4.6]{SasThoWil} when $M$ is simply a polynomial. 

It is not hard to show that for a nontrivial behaviour ${\mathfrak{B}}_{\calD'_{\mA}(\mR^{d+1})}(M)$, one has 
the following hierarchy of properties:
$$
\textrm{controllability } \Rightarrow \textrm{ approximate controllability } 
\Rightarrow \;\; \neg\textrm{(autonomy)}.
$$
Thus in our search for the appropriate algebraic condition characterizing approximate controllability 
of ${\mathfrak{B}}_{\calD'_{\mA}(\mR^{d+1})}(M)$, we expect an algebraic condition 
that lies between the two characterizations 
of controllability and autonomy given in this article in Theorems~\ref{theorem_matricial_autonomy}  
and \ref{result_2}. We leave the investigation of this question for future work.
\end{remark}

\medskip 

\noindent {\bf Acknowledgements:} The author thanks the two anonymous reviewers for several useful comments. 
In particular, for the suggestion of  developing the matrix case results as opposed to only for polynomials 
 done in the previous version of the article.

\end{document}